\documentclass[reqno]{amsart}

\input xy
\xyoption{all}
\usepackage{epsfig}
\usepackage{color}
\usepackage{amsthm}
\usepackage{amssymb}
\usepackage{amsmath}
\usepackage{amscd}
\usepackage{amsopn}
\usepackage{url}

\usepackage{hyperref}\hypersetup{colorlinks}

\usepackage{color} 

\definecolor{darkred}{rgb}{1,0,0} 
\definecolor{darkgreen}{rgb}{0,0.8,0}
\definecolor{darkblue}{rgb}{0,0,1}

\hypersetup{colorlinks,
linkcolor=darkblue,
filecolor=darkgreen,
urlcolor=darkred,
citecolor=darkgreen}

\newcommand{\labell}[1] {\label{#1}}

\numberwithin{equation}{section}
\newtheorem {Theorem}{Theorem}
\numberwithin{Theorem}{section}

\newtheorem {Corollary}[Theorem]{Corollary}
\theoremstyle{definition}

\theoremstyle{remark}
\newtheorem{Remark}[Theorem]{Remark}
\newtheorem{Example}[Theorem]{Example}

\expandafter\chardef\csname pre amssym.def at\endcsname=\the\catcode`\@
\catcode`\@=11
\def\undefine#1{\let#1\undefined}
\def\newsymbol#1#2#3#4#5{\let\next@\relax
 \ifnum#2=\@ne\let\next@\msafam@\else
 \ifnum#2=\tw@\let\next@\msbfam@\fi\fi
 \mathchardef#1="#3\next@#4#5}
\def\mathhexbox@#1#2#3{\relax
 \ifmmode\mathpalette{}{\m@th\mathchar"#1#2#3}%
 \else\leavevmode\hbox{$\m@th\mathchar"#1#2#3$}\fi}
\def\hexnumber@#1{\ifcase#1 0\or 1\or 2\or 3\or 4\or 5\or 6\or 7\or 8\or
 9\or A\or B\or C\or D\or E\or F\fi}

\font\teneufm=eufm10
\font\seveneufm=eufm7
\font\fiveeufm=eufm5
\newfam\eufmfam
\textfont\eufmfam=\teneufm
\scriptfont\eufmfam=\seveneufm
\scriptscriptfont\eufmfam=\fiveeufm

\catcode`\@=\csname pre amssym.def at\endcsname

\newcommand{\CR}{{\mathcal R}}

\def    \C      {{\mathbb C}}
\def    \R      {{\mathbb R}}

\def    \Z      {{\mathbb Z}}
\def    \N      {{\mathbb N}}
\def    \Q      {{\mathbb Q}}

\def    \T      {{\mathbb T}}
\def    \CP     {{\mathbb C}{\mathbb P}}

\def    \12    {{\frac{1}{2}}}
\def    \bD    {\bar{\Delta}}

\def    \codim  {\operatorname{codim}}

\def    \rk     {\operatorname{rk}}

\def    \HF     {\operatorname{HF}}
\def    \HC     {\operatorname{HC}}
\def    \CC     {\operatorname{CC}}

\def    \by     {\bar{y}}

\def    \va     {\vec{a}}
\def    \vb     {\vec{b}}

\def    \MUCZ  {\operatorname{\mu_{\scriptscriptstyle{CZ}}}}
\def    \hMUCZ  {\operatorname{\hat{\mu}_{\scriptscriptstyle{CZ}}}}
\def    \MUM  {\operatorname{\mu_{\scriptscriptstyle{Morse}}}}

\def    \ssminus        {\smallsetminus}

\begin{document}


\setlength{\smallskipamount}{6pt}
\setlength{\medskipamount}{10pt}
\setlength{\bigskipamount}{16pt}





\title[Resonances for Hamiltonian Diffeomorphisms and Reeb Flows]{Homological 
Resonances for Hamiltonian Diffeomorphisms and Reeb Flows}

\author[Viktor Ginzburg]{Viktor L. Ginzburg}
\author[Ely Kerman]{Ely Kerman}

\address{VG: Department of Mathematics, UC Santa Cruz,
Santa Cruz, CA 95064, USA}
\email{ginzburg@math.ucsc.edu}

\address{EK: Department of Mathematics, University of Illinois at 
Urbana-Champaign, Urbana, IL 61801, USA}
\email{ekerman@math.uiuc.edu}

\date{\today} \thanks{The work is partially supported by the NSF and
by the faculty research funds of the University of California, Santa
Cruz.}


\begin{abstract} We show that whenever a Hamiltonian diffeomorphism or
  a Reeb flow has a finite number of periodic orbits, the mean indices
  of these orbits must satisfy a resonance relation,
  provided that the ambient manifold meets some natural 
  requirements. In the case of Reeb flows, this leads to simple expressions
  (purely in terms of the mean indices) for the mean Euler characteristics.
  These are invariants of the underlying contact structure which are capable of
  distinguishing some  contact structures that are  homotopic but not 
  diffeomorphic.
\end{abstract}

\maketitle

\tableofcontents

\section{Introduction and main results }
\labell{sec:main-results}

\subsection{Introduction}
\label{sec:intro}

In this paper, we establish new restrictions on Hamiltonian
diffeomorphisms and Reeb flows which have only finitely many periodic
orbits.  While these dynamical systems are rare, there are many
natural examples, such as irrational rotations of the two-dimensional
sphere and Reeb flows on irrational ellipsoids. Moreover, these
systems serve as important counterpoints to cases where one can prove
the existence of infinitely many periodic orbits, see for example
\cite{E,EH,FrHa,Hi,Gi:conley,GG:gap,GG:gaps,SZ,Vi}. Our main theorems
establish resonance relations for the mean indices of the periodic
orbits of these systems when the ambient manifolds meet some
additional requirements.

For a Hamiltonian diffeomorphism $\varphi$, the additional requirement
on the ambient symplectic manifold is that $N\geq n+1$, where $N$ is
the minimal Chern number and $n$ is half the dimension.  A resonance
relation in this case is simply a linear relation, with integer
coefficients, between the mean indices $\Delta_i$, viewed as elements
of $\R/2N\Z$.  The existence of these relations can be established
essentially by carrying out an argument from \cite{SZ} modulo $2N$,
cf.\ \cite{BCE}.
The specific form of the resonance relations established depends on
$\varphi$, although conjecturally the relation $\sum\Delta_i=0\mod 2N$
is always satisfied.

For Reeb flows, we show that certain sums of the reciprocal mean
indices are equal to the (positive/negative) mean Euler
characteristic, an invariant of the contact structure defined via
cylindrical contact homology, when it itself is defined. (This
relation generalizes the one discovered by Ekeland and Hofer \cite{E,EH} and
Viterbo \cite{Vi} for convex and star-shaped hypersurfaces in
$\R^{2n}$, which served as the main motivation for the present work.)
One can view these resonance relations as new expressions for the mean
Euler characteristics written purely in the terms of the mean
indices.  As is shown below in Example \ref{ex:U}, these invariants
can be used to distinguish some contact structures which are homotopic
but not diffeomorphic, such as those distinguished by Ustilovsky in
\cite{U} using cylindrical contact homology.

One forthcoming application of our results is the $C^\infty$-generic
existence of infinitely many periodic orbits for Hamiltonian
diffeomorphisms or Reeb flows established under the same hypotheses as
the resonance relations; see \cite{GG:generic}. For the Reeb flows,
these results generalize the $C^\infty$-generic existence of
infinitely many closed characteristics on convex hypersurfaces in
$\R^{2n}$ (see \cite{E}) and the $C^\infty$-generic existence of
infinitely many closed geodesics (see \cite{Ra1,Ra2}).

\subsection{Resonances for Hamiltonian diffeomorphisms}
\label{sec:res-hd}
Let $(M^{2n},\omega)$ be a closed symplectic manifold, which throughout
this paper is assumed to be weakly monotone; see, e.g., \cite{HS} or
\cite{MS} for the definition. Denote by $N$ the minimal Chern number
of $(M, \omega)$, i.e., $N$ is the positive generator of the subgroup
$\left<c_1(TM),\pi_2(M)\right>$ of  $\Z$. (When
$c_1(TM)\mid_{\pi_2(M)}=0$, we set
$N=\infty$.) Recall that $(M, \omega)$ is said to be \emph{rational} if the group 
$\left<\omega,\pi_2(M)\right>\subset \R$ is discrete.

A Hamiltonian $H\colon \R/\Z\times M\to \R$ which is one-periodic in
time, determines a vector field $X_H$ on $M$ via Hamilton's equation
$i_{X_H} \omega = -dH$.  Let $\varphi=\varphi_H$ be the Hamiltonian
diffeomorphism of $M$ given by the time-$1$ flow of $X_H$.  Recall
that there is a one-to-one correspondence between $k$-periodic points
of $\varphi$ and $k$-periodic orbits of $H$. In this paper, we
restrict our attention exclusively to periodic points of $\varphi$
such that the corresponding periodic orbits of $H$ are
contractible. (One can show that contractibility of the orbit
$\varphi^t(x)$ of $H$ through a periodic point $x\in M$ is completely
determined by $x$ and $\varphi$ and is independent of the choice of
generating Hamiltonian $H$.)  To such a periodic point $x$, we
associate the mean index $\Delta(x)$, which is viewed here as a point
in $\R/2N\Z$, and hence is independent of the choice of capping of the
orbit. The mean index measures the sum of rotations of the eigenvalues
on the unit circle of the linearized flow $d\varphi^t_H$ along $x$.
The reader is referred to \cite{SZ} for the definition of the mean
index $\Delta$; see also, e.g., \cite{GG:CMH} for a detailed
discussion. We only mention here that
$$
|\Delta(x) - \hMUCZ(x)| \leq n,
$$
where $\hMUCZ(x)$ is the Conley--Zehnder index of $x$, viewed as a
point in $\R/2N\Z$ and given a nonstandard normalization such that for
a critical point $x$ of a $C^2$-small Morse function one has
$\hMUCZ(x) = \MUM(x) -n \mod 2N$.

A Hamiltonian diffeomorphism $\varphi$ is said
to be \emph{perfect} if it has finitely many (contractible) periodic
points, and every periodic point of $\varphi$ is a fixed point. Let
$\Delta_1,\ldots, \Delta_m$ be the collection of the mean indices of the fixed
points of a perfect Hamiltonian diffeomorphism $\varphi$ with exactly
$m$ fixed points.  A \emph{resonance} or  \emph{resonance relation}
is a vector $\va=(a_1,\ldots,a_m)\in\Z^m$ such that
$$
a_1\Delta_1+\ldots+a_m\Delta_m=0\mod 2N.
$$
It is clear that resonances form a free abelian group 
$\CR=\CR(\varphi)\subset \Z^m$. 

\begin{Theorem}
\label{thm:res}
Assume that $n+1\leq N<\infty$. 
\begin{itemize}

\item[(i)] Then $\CR\neq 0$, i.e., the mean indices $\Delta_i$ satisfy at least
  one non-trivial resonance relation.

\item[(ii)] Assume in addition that there is only one resonance, i.e.,
  $\rk\CR=1$, and let $\va=(a_1,\ldots,a_m)$ be a generator of $\CR$
  with at least one positive component. Then all components of $\va$ are
  non-negative, i.e., $a_i\geq 0$ for all $i$, and
$$
\sum a_i \leq \frac{N}{N-n}.
$$

\item[(iii)] Furthermore, assume that $(M, \omega)$ is rational. Then assertion
  (i) holds when only irrational mean indices are considered (i.e.,
  the irrational mean indices satisfy a non-trivial resonance
  relation) and assertion (ii) holds when only non-zero mean indices
  are considered.
\end{itemize}

\end{Theorem}

We require $(M, \omega)$ to be weakly monotone here only for the sake of
simplicity: this condition can be eliminated by utilizing the
machinery of virtual cycles. Likewise, the hypothesis that $(M, \omega)$ is
rational in (iii) is purely technical and probably
unnecessary. However, the proof of (iii) relies on a result from
\cite{GG:gaps} which has so far been established only for rational
manifolds although one can expect it to hold without this requirement;
see \cite[Remark 1.19]{GG:gaps}.  When $N=\infty$, i.e.,
$c_1(TM)\mid_{\pi_2(M)}=0$, perfect Hamiltonian diffeomorphisms probably
do not exist and the assertion of the theorem is void. For instance,
if $(M, \omega)$ is rational and $N=\infty$, every Hamiltonian
diffeomorphism has infinitely many periodic points; see
\cite{GG:gaps}.

We note that every $\Delta_i\in\Q$ (e.g., $\Delta_i=0$)
automatically gives rise to an infinite cyclic subgroup of
resonances. Thus, assertion (iii) is much more precise than (i) or (ii) in the
presence of rational or zero mean indices.

Finally we observe that the condition that $\varphi$ is perfect can be
relaxed and replaced by the assumption that $\varphi$ has finitely
many periodic points. Indeed, in this case, suitable iterations
$\varphi^k$ are perfect. Applying Theorem \ref{thm:res} to such a 
$\varphi^k$, we then obtain resonance relations involving (appropriately
normalized) mean indices of all periodic points of $\varphi$.

\begin{Example} 
\label{ex:CPn}
Let $\varphi$ be the Hamiltonian diffeomorphism of $\CP^n$ generated
by a quadratic Hamiltonian
$H(z)=\pi\big(\lambda_0|z_0|^2+\ldots+\lambda_n|z_n|^2\big)$, where
the coefficients $\lambda_0,\ldots,\lambda_n$ are all distinct. (Here,
we have identified $\CP^n$ with the quotient of the unit sphere in
$\C^{n+1}$. Recall also that $N=n+1$ for $\CP^n$.)  Then, the
Hamiltonian diffeomorphism $\varphi=\varphi_H$ is perfect and has
exactly $n+1$ fixed points (the coordinate axes). The mean indices are
$$
\Delta_i=\sum_j\lambda_j-(n+1)\lambda_i,
$$
where now $i=0,\ldots,n$. Thus, $\sum \Delta_i=0$ and this is the
only resonance relation for a generic choice of the coefficients: the
image of the map $(\lambda_0,\ldots,\lambda_n)\mapsto
(\Delta_0,\ldots,\Delta_n)$ is the hyperplane  $\sum \Delta_i=0$.

\end{Example}

More generally, we have

\begin{Example} 
\label{ex:Ham-gp-action}
Suppose that $(M, \omega)$ is equipped with a Hamiltonian torus action
with isolated fixed points; see, e.g., \cite{GGK,MS:intro} for the
definition and further details. A generic element of the torus gives
rise to a perfect Hamiltonian diffeomorphism $\varphi$ of $(M,
\omega)$ whose fixed points are exactly the fixed points of the torus
action.  One can show that in this case the mean indices again satisfy the resonance
relation $\sum \Delta_i=0$. (The authors are grateful to Yael Karshon
for a proof of this fact; \cite{Ka}.)

Examples of symplectic manifolds which admit such torus actions include the
majority of coadjoint orbits of compact Lie groups. One can also
construct new examples from a given one by equivariantly blowing-up
the symplectic manifold at its fixed points. The resulting symplectic
manifold always inherits a Hamiltonian torus action and, in many
instances, this action also has isolated fixed points.
\end{Example}

These examples suggest that, in the setting of Theorem \ref{thm:res},
the mean indices always satisfy the resonance relation $\sum
\Delta_i=0$. The next result can be viewed as preliminary step
towards proving this conjecture. 

\begin{Corollary}
\label{cor:CPn}
Let $\varphi$ be a perfect Hamiltonian diffeomorphism of $\CP^n$ such
that there is only one resonance, i.e., $\rk\CR=1$. Denote by $\va$ a
generator of $\CR$ as described in statement (ii) of Theorem
\ref{thm:res}, and assume that $a_i\neq 0$ for all $i$.  Then
$\varphi$ has exactly $n+1$ fixed points and $\va=(1,\ldots,1)$, i.e.,
the mean indices satisfy the resonance relation $\sum\Delta_i=0$.
\end{Corollary}

\begin{proof}
  By the Arnold conjecture for $\CP^n$, we have $m\geq n+1$, see
  \cite{Fo,FW} and also \cite{F:c-l,Sc2}. By (ii), $a_i\neq 0$ means
  that $a_i\geq 1$. Hence, by (ii) again, $m=n+1$ and $a_i=1$ for all
  $i$ since $N=n+1$ and $\sum a_i\leq N/(N-n)=n+1$.
\end{proof}

Conjecturally, any Hamiltonian diffeomorphism of $\CP^n$
with more than $n+1$ fixed points has infinitely many periodic
points. (For $n=1$ this fact is established in \cite{FrHa}.) Corollary
\ref{cor:CPn} implies that that this is indeed the case, provided that
the mean indices satisfy exactly one resonance relation (i.e., $\rk
\CR=1$) and all components of the resonance relation are non-zero. (We
emphasize that by Theorem \ref{thm:res}, $\rk\CR\geq 1$.)

\begin{Remark}
 The resonances considered here are not the only numerical constraints 
 on  the fixed points of a perfect Hamiltonian diffeomorphism
  $\varphi\colon M\to M$. Relations of a different type,
  involving both the mean indices and action values, are established in
  \cite{GG:gaps} when $(M, \omega)$ is either monotone or negative
  monotone. For instance, it is proved there that the so-called augmented
  action takes the same value on all periodic points of $\varphi\colon
  \CP^n\to \CP^n$ whenever $\varphi$ has exactly $n+1$ periodic points.

  Note also that a perfect Hamiltonian diffeomorphism need not be
  associated with a Hamiltonian torus action as are the Hamiltonian
  diffeomorphisms in Examples \ref{ex:CPn} and
  \ref{ex:Ham-gp-action}. For instance, there exists a Hamiltonian
  perturbation $\varphi$ of an irrational rotation of $S^2$ with
  exactly three ergodic invariant measures: the Lebesgue measure and
  the two measures corresponding to the fixed points of $\varphi$;
  \cite{AK,FK}. Clearly, $\varphi$ is perfect and not conjugate to a
  rotation.
\end{Remark}

\begin{Remark}
  As is immediately clear from the proof, one can replace in Theorem
  \ref{thm:res} the collection $\Delta_1,\ldots,\Delta_m$ of the mean
  indices of the fixed points of $\varphi$ by the set of all distinct mean
  indices.  This is a refinement of the theorem, for an equality of
  two mean indices is trivially a resonance relation. Note also that,
  as a consequence of this refinement, all mean indices are distinct
  whenever $\rk \CR=1$.
\end{Remark}

\subsection{Resonances for Reeb flows} Let $(W^{2n-1},\xi)$ be a
closed contact manifold such that the \emph{cylindrical contact homology}
$\HC_*(W,\alpha)$ is defined. More specifically, we require $(W,\xi)$ to
admit a contact form $\alpha$ such that 
\begin{itemize}

\item[(CF1)] all periodic orbits of the Reeb flow of $\alpha$ are
  non-degenerate, and

\item[(CF2)] the Reeb flow of $\alpha$ has no contractible periodic orbits $x$ with
$|x|=\pm 1$ or $0$.

\end{itemize}
Here, $|x|=\MUCZ(x)+n-3$, where $\MUCZ(x)$ stands for the
Conley--Zehnder index of $x$ (with its standard normalization). For
the sake of simplicity, we also assume that
$c_1(\xi)=0$. Then $\HC_*(W,\xi)$ is the
homology of a complex $\CC_*(W,\alpha)$ which is generated (over a
fixed ground field, say, $\Z_2$) by certain periodic orbits of the
Reeb flow, and is graded via $|\cdot|$. To be more precise, the
generators of $\CC_*(W,\alpha)$ are all iterations of good Reeb orbits
and odd iterations of bad Reeb orbits (See the definitions below.)
The homology $\HC_*(W,\xi)$ is independent of $\alpha$ as long as
$\alpha$ meets requirements (CF1) and (CF2). The exact nature of the
differential on $\CC_*(W,\alpha)$ is inessential for our
considerations. We refer the reader to, for instance, \cite{Bo,BO,El}
and the references therein for a more detailed discussion of contact
homology.

Furthermore, assume that 
\begin{itemize}
\item[(CH)]  there are two integers $l_+$ and $l_-$, such that  the space 
$\HC_l(W,\xi)$ is finite-dimensional for
$l\geq l_+$ and $l\leq l_-$.  
\end{itemize}
In the examples considered here, the contact homology is finite
dimensional in all degrees and this condition is automatically
met. By analogy with the constructions from \cite{EH,Vi}, we set
\begin{equation}
\label{eq:Euler1}
\chi^\pm(W,\xi)=\lim_{N\to\infty}\frac{1}{N}
\sum_{l=l_\pm}^N(-1)^l\dim \HC_{\pm l}(W,\xi),
\end{equation}
provided that the limits exist.
Clearly, when $\HC_l(W,\xi)$ is finite--dimensional for all $l$, we have
$$
\frac{\chi^+(W,\xi)+\chi^-(W,\xi)}{2}=\chi(W,\xi)
:=
\lim_{N\to\infty}\frac{1}{2N+1}
\sum_{l=-N}^N(-1)^l\dim \HC_{l}(W,\xi).
$$
We call $\chi^\pm(W,\xi)$ the \emph{positive/negative mean Euler
characteristic} of $\xi$. Likewise, we call $\chi(W,\xi)$   the mean
Euler characteristic of $\xi$. (This invariant is also considered in
\cite[Section 11.1.3]{VK:thesis}.)

In what follows, we denote by $x^k$ the $k$th iteration of a periodic
orbit $x$ of the Reeb flow of $\alpha$ on $W$.  Recall that a simple
periodic orbit $x$ is called \emph{bad} if the linearized Poincar\'e
return map along $x$ has an odd number of real eigenvalues strictly
smaller than $-1$. Otherwise, the orbit is said to be
\emph{good}. (This terminology differs slightly from the standard
usage, cf.\ \cite{Bo,BO}.)  When the orbit $x$ is good, the parity of
the Conley--Zehnder indices $\MUCZ(x^k)$ is independent of $k$; if $x$
is bad, then the parity of $\MUCZ(x^k)$ depends on the parity of $k$.

To proceed, we now assume that
\begin{itemize}
\item[(CF3)] the Reeb flow of $\alpha$ has finitely many simple periodic orbits.
\end{itemize}
In contrast with (CF1) and (CH) and even (CF2), this is a very strong
restriction on $\alpha$. Denote the good simple periodic orbits of the
Reeb flow by $x_i$ and the bad simple periodic orbits by $y_i$. Then,
$\CC_*(W,\alpha)$ is generated by the $x_i^k$, for all $k$, together
with the $y_i^k$ for $k$ odd.  Whenever (CF3) holds, condition (CH) is
automatically satisfied with $l_-=-2$ and $l_+=2n-4$ . Moreover, in
this case the spaces $\CC_l(W,\alpha)$ are finite--dimensional. (This
fact can, for instance, be extracted from the proof of Theorem
\ref{thm:reeb}; see \eqref{eq:index1} and \eqref{eq:index2}.)
Likewise, all spaces $\CC_l(W,\alpha)$ (and hence $\HC_l(W,\xi)$) are
finite--dimensional, provided that all of the orbits $x_i$ and $y_i$
have non-zero mean indices. We denote the mean index of an orbit $x$
by $\Delta(x)$ and set
$\sigma(x)=(-1)^{|x|}=-(-1)^n(-1)^{\MUCZ(x)}$. In other words,
$\sigma(x)$ is, up to the factor $-(-1)^n$, the topological index of
the orbit $x$ or, more precisely, of the Poincar\'e return map of
$x$. 

\begin{Theorem}
\label{thm:reeb}
Assume that $\alpha$ satisfies conditions (CF1)--(CF3). Then the limits in
\eqref{eq:Euler1} exist and 
\begin{equation}
\label{eq:Euler2}
{\sum}^{\pm} \frac{\sigma(x_i)}{\Delta(x_i)}
+\frac{1}{2}{\sum}^{\pm} \frac{\sigma(y_i)}{\Delta(y_i)}=
\chi^{\pm}(W,\xi),
\end{equation}
where ${\sum}^{+}$ (respectively, ${\sum}^{-}$) stands for the sum over
all orbits with positive (respectively, negative) mean index.
\end{Theorem}

This theorem will be proved in Section \ref{sec:contact}. Here we only
mention that the specific nature of the differential on the complex
$\CC_*(W,\xi)$ plays no role in the argument. Also note that a similar
result holds when the homotopy classes of orbits are restricted to any
set of free homotopy classes closed under iterations, provided that
(CF1)--(CF3) hold for such orbits. For instance, \eqref{eq:Euler2}
holds when only contractible periodic orbits are taken into account in
the calculation of the left-hand side and the definition of
$\chi^\pm$.  Also it is worth pointing out that for non-contractible
orbits the definitions of the Conley--Zehnder and mean indices involve
some additional choices (see, e.g., \cite{Bo}) which effect both the
right- and the left-hand side of~\eqref{eq:Euler2}. 

\begin{Example}
\label{ex:standard}
Let $\xi_0$ be the standard contact structure on
$S^{2n-1}$. Then, as is easy to see, $\chi^-(S^{2n-1},\xi_0)=0$ and
$\chi^+(S^{2n-1},\xi_0)=1/2$. In this case, the resonance relations
\eqref{eq:Euler2} were proved in \cite{Vi}. (The case of a convex
hypersurface in $\R^{2n}$ was originally considered in \cite{E,EH}.)
\end{Example}

By definition \eqref{eq:Euler1}, the mean Euler characteristics
$\chi^\pm(W, \xi)$ are invariants of the contact structure
$\xi$. (Strictly speaking this is true only when $(W,\xi)$ is equipped
with some extra data or $W$ is simply connected.) Theorem
\ref{thm:reeb} implies that, whenever there is a contact form $\alpha$
for $\xi$ which satisfies conditions (CF1)--(CF3), these invariants
can, in principle, be calculated by purely elementary means (without
first calculating the contact homology) via the mean indices of closed
Reeb orbits. The following example shows that the mean Euler
characteristics can distinguish some non-diffeomorphic contact
structures within the same homotopy class.

\begin{Example}
\label{ex:U}
In \cite{U}, Ustilovsky considers a family of contact structures
$\xi_p$ on $S^{2n-1}$ for odd $n$ and positive $p\equiv \pm 1\mod 8$.
For a fixed $n$, the contact structures $\xi_p$ fall within a finite
number of homotopy classes, including the class of the standard
structure $\xi_0$.  By computing $\HC_*(S^{2n-1},\xi_p)$, Ustilovsky
proves that the structures $\xi_p$ are mutually non-diffeomorphic, and
that none of them are diffeomorphic to $\xi_0$.

It follows from \cite{U}, that $\xi_p$ can be given by a contact
form $\alpha_p$ satisfying conditions (CF1)--(CF3). Furthermore, it is
not hard to show (see below or \cite{VK:thesis}) that
\begin{equation}
\label{eq:ust}
\chi^+(S^{2n-1},\xi_p)=\frac{1}{2}
\left(
\frac{p(n-1)+1}{p(n-2)+2}
\right)
\end{equation}
and $\chi^-(S^{2n-1},\xi_p)=0$. The right-hand side of \eqref{eq:ust}
is a strictly increasing function of $p>0$.  Hence, the positive mean
Euler characteristic distinguishes the structures $\xi_p$ with $p>0$.
Note also that $\chi^+(\xi_p)>\chi^+(S^{2n-1},\xi_0)=1/2$ when $p>1$
and $\chi^+(\xi_1)=1/2$. In particular, $\chi^+$ distinguishes $\xi_p$
with $p>1$ from the standard structure $\xi_0$.

Formula \eqref{eq:ust} can be established in two ways, both relying on
\cite{U}.  The first way is to use the contact form $\alpha_p$
constructed in \cite{U}.  The indices of periodic orbits of $\alpha_p$
are determined in \cite{U} and the mean indices can be found in a
similar fashion or obtained using the asymptotic formula
$\Delta(x)=\lim_{k\to\infty} \MUCZ(x^k)/k$; see \cite{SZ}.  Then
\eqref{eq:Euler2} is applied to calculate
$\chi^+(S^{2n-1},\xi_p)$. (Note that this calculation becomes even
simpler when the Morse--Bott version of \eqref{eq:Euler2} is used,
reducing the left-hand side of \eqref{eq:Euler2} to just one
term for a suitable choice of contact form, see \cite{Es}.)
Alternatively one can use the definition \eqref{eq:Euler1} of $\chi^+$
and the calculation of $\dim \HC_{*}(S^{2n-1},\xi_p)$ from \cite{U};
see \cite[Section 11.1.3]{VK:thesis}.
\end{Example}

\begin{Remark} 
\label{rmk:filling}
A different version of contact homology, the linearized
  contact homology, is defined when $(W, \xi)$ is equipped with a
  symplectic filling $(M, \omega)$. The chain group for this homology
  is still described via the closed orbits of a Reeb flow for
  $\xi$. In particular, it is still generated by all iterations of
  good orbits and the odd iterations of bad ones. The differential is
  defined via the augmentation, associated with $(M, \omega)$, on the
  full contact homology differential algebra; see, e.g., \cite{BO,El}.
  Thus, the linearized contact homology, in general, depends on $(M,
  \omega)$.
  
  A key point in this construction is that one does not need to assume
  condition (CF2) in order to define the linearized contact homology.
  (When (CF2) is satisfied, the linearized contact homology coincides
  with the cylindrical contact homology.)  Hence, for fillable contact
  manifolds, Theorem \ref{thm:reeb} holds without this assumption.
  This follows immediately from the fact that our proof of the theorem
  makes no use of the specific nature of the differential.  One still
  requires the contact form $\alpha$ to satisfy (CF1) and (CF3),
  and the filling $(M, \omega)$ is assumed to be such that
  $\left<\omega,\pi_2(M)\right>=0$ and $c_1(TM)=0$.
\end{Remark}

\begin{Remark} An argument similar to the proof of Theorem
  \ref{thm:reeb} also establishes the following ``asymptotic Morse
  inequalities''
$$
{\sum}^{\pm}\frac{1}{\Delta(x_i)}
+\frac{1}{2}{\sum}^{\pm}\frac{1}{\Delta(y_i)}
\geq \limsup_{N\to\infty}\frac{1}{N}\sum_{l=l_\pm}^N\dim \HC_{\pm l}(W,\xi),
$$
provided that $\alpha$ satisfies (CF1)--(CF3) or in the setting of
Remark \ref{rmk:filling}. A similar inequality (as well as some other
relations between mean indices and actions) is proved in \cite{EH} for
convex hypersurfaces in $\R^{2n}$.
\end{Remark}

\begin{Remark} 
  Finally note that the quite restrictive requirement that the Reeb
  flow is non-degenerate and has finitely many periodic orbits can be
  relaxed in a variety of ways. For instance, the Morse--Bott version
  of Theorem \ref{thm:reeb} is proved in \cite{Es}, generalizing the
  resonance relations for geodesic flows established in \cite{Ra1}.

\end{Remark}

\subsection{Acknowledgments} The authors are grateful to Yasha
Eliashberg, Jacqui Espina, Ba\c sak G\"urel, Yael Karshon, and Anatole Katok for
useful discussions. The authors also wish to thank the referee for her/his
comments and remarks.

\section{Resonances in the Hamiltonian case}

\subsection{Resonances and subgroups of the torus}
\label{sec:res-gen}
Consider a closed subgroup $\Gamma$ of $ \T^m=\R^m/\Z^m$ which is
topologically generated by an element $\bD=(\bD_1,\ldots,\bD_m)\in
\T^m$. In other words, $\Gamma$ is the closure of the set $\{k\bD\mid
k\in\N\}$, which we will call the orbit of $\bD$. Note that $\Gamma$
is a Lie group since it is a closed subgroup of a Lie group. (We
refer the reader to, e.g., \cite{DK,Ki} for the results on Lie groups and duality
used in this section.)  Moreover, the connected component of the
identity $\Gamma_0$ in $\Gamma$ is a torus, for $\Gamma_0$ is compact,
connected and abelian.  Denote by $\CR$ the group of characters
$\T^m\to S^1$ which vanish on $\Gamma$ or equivalently on $\bD$. Thus,
$\CR$ is a subgroup of the dual group $\Z^m$ of $\T^m$.  We can think
of $\CR$ as the set of linear equations determining $\Gamma$. In other
words, $\va$ belongs to $\CR$ if and only if
$$
a_1\bD_1+\ldots + a_m\bD_m=0 \mod 1 .
$$
We will refer to $\CR$ as the group of resonances associated to
$\Gamma$.  Clearly, $\Gamma$ is completely determined by $\CR$.
When the role of $\bD$ or $\Gamma$ needs to be emphasized,
we will use the notation $\Gamma(\bD)$ and $\CR(\Gamma)$ or
$\CR(\bD)$, etc.  Furthermore, we denote by $\CR_0\supset \CR$ the group of
resonances associated to $\Gamma_0$.

We will need the following properties of $\Gamma$ and $\CR$:
\begin{itemize}

\item $\codim \Gamma=\rk \CR$;
\item $\Gamma\subset \Gamma'$ iff $\CR(\Gamma)\supset \CR(\Gamma')$;
\item $\Gamma/\Gamma_0$ and $\CR_0/\CR$ are finite cyclic groups dual to each other.
\end{itemize}
Here the second assertion is obvious. To prove the first and the last
ones, first note that $\Gamma/\Gamma_0$ is finite and cyclic since
$\Gamma$ is compact and has a dense cyclic subgroup. Further note that
$\CR$ can be identified with the dual group of $\T^m/\Gamma$ and that
the first assertion is clear when $\Gamma=\Gamma_0$, i.e., $\Gamma$ is
a torus.  Dualizing the exact sequence $0\to
\Gamma/\Gamma_0\to \T^m/\Gamma_0\to\T^m/\Gamma \to 0$ we obtain the
exact sequence $0\to\CR\to\CR_0\to \CR_0/\CR\to 0$ (see, e.g.,
\cite[Chapter 12]{Ki}) and the last assertion follows. (The dual of a
finite cyclic group, say $\Z_k$, is isomorphic to $\Z_k$.) It also
follows that $\rk \CR=\rk \CR_0=\codim \Gamma_0=\codim \Gamma$.

\subsection{Proof of Theorem \ref{thm:res}} Let $\Delta = (\Delta_1,
\ldots, \Delta_m)$ where the components $\Delta_i \in \R/2N\Z$ are the mean
indices of the $m$ periodic points of the perfect Hamiltonian
diffeomorphism $\varphi$.  Set $\bD=\Delta/2N$. Then, $\bD$ belongs to
$\T^m$ and we have $\CR(\varphi)=\CR(\bD)$. Recalling that $n<N$, we
define $\Pi$ to be the cube $(n/N,1)^m$ in $\T^m$. In other words, $\Pi$
consists of points $\theta=(\theta_1,\ldots,\theta_m)\in\T^m$ such
that $\theta_i$ is in the arc $(n/N,1)$ for all $i$. We will refer to
$\Pi$ as the \emph{prohibited region} of $\T^m$.

\subsubsection{Proof of (i)} By a standard argument, for
every $k\in \N$ at least one component of $k\bD$ is in the arc
$[0,n/N]$. (See \cite{SZ} or, e.g., \cite{GG:gaps}; we will also
briefly recall the argument in the proof of (iii) below.) In other
words, none of the points of the orbit $\{k\bD\mid k\in\N\}$ lies in
the prohibited region $\Pi$. Since $\Pi$ is open, we conclude that
$\Gamma\cap\Pi=\emptyset$. Hence, $\codim \Gamma>0$ and $\CR\neq 0$.

\subsubsection{Proof of (ii)} Here we assume that $\rk\CR=1$. Let
$\va$ be a generator of $\CR$. Then,
$\Gamma$ is given by the equation
$$
\va\cdot\theta:=\sum_i a_i \theta_i=0 \text{ in } \R/\Z,
$$
where $\theta =(\theta_1,\ldots,\theta_m)\in\T^m$. Note that $\Pi$ can
also be viewed as the product of the arcs $(-1+n/N,0)$ in $\T^m$. Thus
the intersection of $\Pi$ with a neighborhood of $(0, \ldots, 0) \in
\T^m$ fills in the (open) portion of the negative quadrant in that
neighborhood. Since $\Gamma\cap \Pi=\emptyset$, all non-zero
components $a_i$ must have the same sign. Hence, if $\va$ has at least
one positive component we have $a_i\geq 0$.

Let $L=\{ (t,\ldots,t)\mid t\in S^1\}$ be the ``diagonal''
one-parameter subgroup of $\T^m$. The point of $L$ with $t=-1/\sum
a_i$ lies in $\Gamma$. Hence, this point must be outside $\Pi$ and so
$|t|\geq 1-n/N$.  It follows that $\sum a_i\leq N/(N-n)$.

\subsubsection{Proof of (iii)} The proofs of (i) and (ii) above are
based on the observation that for every $k$, there exists a capped
$k$-periodic orbit $\by$ such that the local Floer homology of
$\varphi^k$ at $\by$ is non-zero in degree $n$:
$\HF_n(\varphi^k,\by)\neq 0$. Then $\Delta(\by) \in [0,\,2n]$. (See
\cite{Gi:conley,GG:gap,GG:gaps} for the proofs of these facts; however, the
argument essentially goes back to \cite{SZ}. Note also that here we
use the grading of the Floer homology by $\hMUCZ$,
i.e., the fundamental class has degree $n$.)  The orbit $y$ is the
$k$th iteration of some orbit $x_i$, and hence $\Delta(y)=k\Delta_i$ in
$\R / 2N \Z$.  We claim that necessarily $\Delta_i\neq 0$, provided that
$M$ is rational and, as before, $N\geq n+1$. As a consequence, the
orbits with $\Delta_i=0$ can be discarded in the proofs of (i) and
(ii).

To show that $\Delta_i\neq 0$, we argue by contradiction. Assume
the contrary: $\Delta_i=0$. Then $\Delta(\by)=0\mod 2N$ and, in fact,
$\Delta(\by) =0$, since we also have $\Delta(\by) \in [0,\,2n]$ and
$N\geq n+1$. The condition that $\Delta(\by) =0$ and
$\HF_n(\varphi^k,\by)\neq 0$ is equivalent to that $\by$ is a
symplectically degenerate maximum of $\varphi^k$; \cite{GG:gap,GG:gaps}. By
\cite[Theorem 1.18]{GG:gaps}, a Hamiltonian diffeomorphism with
symplectically degenerate maximum necessarily has infinitely many
periodic points whenever $M$ is rational. This contradicts the
assumption that $\varphi$ is perfect.

Thus, we have proved that (i) and (ii) hold with only non-zero mean
indices (in $\R / 2N \Z$) taken into account. To finish the proof of
(iii), it suffices to note that replacing $\varphi$ by $\varphi^k$, for
a suitably chosen $k$, we can make every rational mean index
zero. Since every resonance relation for $\varphi^k$ is also a
resonance relation for $\varphi$, we conclude that the irrational mean
indices of $\varphi$ satisfy a resonance relation.

\begin{Remark} The requirement that $M$ be rational enters the proof
  of (iii) only at the last point where \cite[Theorem 1.18]{GG:gaps}
  is utilized.  The role of this requirement in the proof of this
  theorem is purely technical and it is likely that the requirement
  can be eliminated. Note also that we do not assert that (ii) holds
  when only irrational mean indices are considered. However, an
  examination of the above argument shows that the following is
  true. Assume that the resonance group for the irrational mean
  indices has rank one for a perfect Hamiltonian diffeomorphism
  $\varphi$. Then these mean indices satisfy a non-trivial resonance
  relation of the form $r\vb$, where $r$ is a natural number, $b_i\geq
  0$ for all $i$, and $\sum b_i\leq N/(N-n)$.
\end{Remark}

\subsection{Perfect Hamiltonian flows on $\CP^n$.}
In this section, we state (without proof) another result asserting,
roughly speaking, that $\sum \Delta_i=0$ for perfect flows on $\CP^n$
satisfying some additional, apparently generic, requirements.

Consider an autonomous, Morse Hamiltonian $H$ on $\CP^n$ with
exactly $n+1$ critical points $x_0,\ldots, x_n$.  Let us call
$\tau\in\R$, $\tau> 0$, a \emph{critical} period if at least one of the
critical points of $H$ is degenerate when viewed as a $\tau$-periodic
orbit of $H$ or equivalently as a fixed point of $\varphi^\tau_H$. We
denote the collection of critical times by $C_H\subset \R$ and call
$t\in (0,\infty)\ssminus C_H$ \emph{regular}.  Assume furthermore that
for every regular $t>0$ the points $x_0,\ldots, x_n$ are the only
fixed points of $\varphi_H^t$ and that for every critical time
$\tau>0$ at least one of the points $x_i$ is non-degenerate as a
fixed point of $\varphi_h^{\tau}$. \emph{Then $\sum \Delta(x_i,t)=0$ for
any regular $t>0$, where $\Delta(x_i,t)$ is the mean index of $\varphi^t_H$
at $x_i$ equipped with trivial capping.}
  
In particular, $\sum\Delta_i=0$ in the setting of
Theorem \ref{thm:res} with $\varphi=\varphi^t_H$. (To ensure that
$\varphi$ satisfies the hypotheses of the theorem it suffices to
require that $kt\not\in C_H$ for all $k\in\N$.)  A quadratic
Hamiltonian on $\CP^n$ with $n\geq 2$ from Example \ref{ex:CPn}
meets the above conditions for generic eigenvalues $\lambda_i$ or,
more precisely, if and only if $H$ generates a Hamiltonian action of a
torus of dimension greater than one.

The proof of this result, to be detailed elsewhere, goes beyond the
scope of the present paper. The argument is conceptually similar to the proof of
\cite[Theorem 1.12]{GG:gaps} but is technically more involved, for it
relies on a more delicate version of Ljusternik--Schnirelman theory
than the one considered in that paper.

\section{Reeb flows: the proof of Theorem \ref{thm:reeb}}
\label{sec:contact}

Our goal in this section is to prove Theorem \ref{thm:reeb}. We focus
on establishing the result for $\chi^+(W,\xi)$. The case of
$\chi^-(W,\xi)$ can be handled in a similar fashion.

First recall that for every periodic orbit $x$ of the Reeb flow, we have
\begin{equation}
\label{eq:index1}
|\MUCZ(x^k)-k\Delta(x)|<n-1
\end{equation}
(see, e.g., \cite{SZ}), and hence
\begin{equation}
\label{eq:index2}
-2<|x^k|-k\Delta(x)<2n-4.
\end{equation}
In particular, it follows from \eqref{eq:index2} that condition (CF3)
implies condition (CH) with $l_+=2n-4$ and $l_-=-2$. Moreover, the
dimension of $\CC_l(W,\alpha)$ with $l\geq l_+$ or $l\leq l_-$ is
bounded from above by a constant independent of $l$.

To simplify the notation, let us set 
$C_l:=\CC_l(W,\alpha)$. Denote by $C_*^{(N)}$ the complex $C_*$ truncated
from below at $l_+$ and  from above at $N>l_+$. In other words, 
$$
C_l^{(N)}=
\begin{cases}
C_l &\text{when  $l_+\leq l\leq N$,}\\
0 &\text{otherwise.}
\end{cases}
$$
The complex $C_*^{(N)}$ is generated by the iterations $x_i^k$ and
$y_i^k$ (with odd $k$) such that $|x_i^k|$ and $|y_i^k|$ are in the
range $[l_+,\,N]$. Note that by \eqref{eq:index2} this can only happen
when $\Delta(x_i)>0$ and $\Delta(y_i)>0$. By \eqref{eq:index2} again,
an orbit $x_i^k$ or $y_i^k$ with positive mean index $\Delta$ is in
$C_*^{(N)}$ for $k$ ranging from some constant (depending on the
orbit, but not on $N$) to roughly $N/\Delta$, up to a constant
independent of $N$.  Furthermore, the parity of $|x_i^k|$ and
$|y_i^k|$ (odd $k$) is independent of $k$, i.e.,
$\sigma(x_i^k)=\sigma(x_i)$ and $\sigma(y_i^k)=\sigma(y_i)$. Thus, the
contribution of the iterations of $x_i$ to the Euler characteristic
$$
\chi\big(C_*^{(N)}\big):=\sum(-1)^l\dim C_l^{(N)}=\sum_{l=l_+}^N(-1)^l\dim C_l
$$
is $\sigma(x_i)N/\Delta(x_i)+ O(1)$ as $N\to\infty$. Likewise, the
contribution of the iterations of $y_i$ is $\sigma(y_i)N/2\Delta(y_i)+
O(1)$, since $k$ assumes only odd values in this case.  Summing up over
all $x_i$ and $y_i$ with positive mean index, we have
$$
\chi\big(C_*^{(N)}\big)=N\Bigg(
{\sum}^+\frac{\sigma(x_i)}{\Delta(x_i)}
+\frac{1}{2}{\sum}^+\frac{\sigma(y_i)}{\Delta(y_i)}
\Bigg)
+ O(1),
$$
and hence
$$
\lim_{N\to\infty}\frac{\chi\big(C_*^{(N)}\big)}{N}
=
{\sum}^+\frac{\sigma(x_i)}{\Delta(x_i)}
+\frac{1}{2}{\sum}^+\frac{\sigma(y_i)}{\Delta(y_i)}.
$$

To finish the proof it remains to show that 
\begin{equation}
\label{eq:chi}
\chi^+(W,\xi)=\lim_{N\to\infty} \chi\big(C_*^{(N)}\big)/N,
\end{equation}
which is nearly obvious. Indeed, 
by the very definition of $C_*^{(N)}$, we have 
$H_l\big(C_*^{(N)}\big)=\HC_l(W,\xi)$ when $l_+<l<N$. Furthermore,
$|H_N\big(C_*^{(N)}\big)-\HC_N(W,\xi)|=O(1)$ since $\dim C_N=O(1)$.
Hence, 
$$
\chi\big(C_*^{(N)}\big)=\sum_l(-1)^l\dim H_l\big(C_*^{(N)}\big)
=\sum_{l=l_+}^N (-1)^l\dim\HC_l(W,\xi)+O(1)
$$
and \eqref{eq:chi} follows. This completes the proof of the theorem.


\begin{thebibliography}{CFHWI}

\bibitem[AK]{AK}
D.V.  Anosov, A.B. Katok,  New examples in smooth ergodic
   theory. Ergodic diffeomorphisms, (in Russian), 
\emph{Trudy Moskov.\ Mat.\ Ob\v{s}\v{c}.},
\textbf{23} (1970), 3--36. 


\bibitem[Bo]{Bo}
F. Bourgeois, Introduction to contact homology. Lecture notes available at
\url{http://homepages.vub.ac.be/~fbourgeo/}.

\bibitem[BCE]{BCE}
F. Bourgeois, K. Cieliebak, T. Ekholm,
A note on Reeb dynamics on the tight 3-sphere, \emph{J. Mod.\ Dyn.},  \textbf{1} (2007), 
597--613.

\bibitem[BO]{BO}
F. Bourgeois, A. Oancea,  	
An exact sequence for contact- and symplectic homology, \emph{Invent.\ Math.},
\textbf{175} (2009), 611--680.


\bibitem[DK]{DK} J.J. Duistermaat, J.A.C. Kolk, \emph{Lie Groups},
  Springer-Verlag, Berlin, New York, 2000.


\bibitem[Ek]{E}
I. Ekeland, Une th\'eorie de Morse pour les syst\`emes hamiltoniens convexes,
\emph{Ann.\ Inst.\ H. Poincar\'e Anal.\ Non Lin\'eaire},  \textbf{1}  (1984), 19--78.

\bibitem[EH]{EH}
I. Ekeland, H. Hofer,  Convex Hamiltonian energy surfaces and their
periodic trajectories,  \emph{Comm.\ Math.\ Phys.},  \textbf{113}  (1987),  
419--469.


\bibitem[El]{El}
Y. Eliashberg, Symplectic field theory and its applications,  \emph{International 
Congress of Mathematicians.}, Vol.\ I,  217--246, Eur.\ Math.\ Soc., Z\"urich, 2007. 

\bibitem[Es]{Es}
J. Espina, Work in progress.


\bibitem[FK]{FK}
B. Fayad, A. Katok,
Constructions in elliptic dynamics, 
\emph{Ergodic Theory Dynam.\ Systems}, \textbf{24} (2004), 1477--1520.


\bibitem[Fl]{F:c-l}
A. Floer,
Cuplength estimates on Lagrangian intersections, \emph{Comm.\ Pure
Appl.\ Math.}, \textbf{42} (1989), 335--356.

\bibitem[Fo]{Fo}
B. Fortune, 
A symplectic fixed point theorem for $\CP^n$, \emph{Invent.\ Math.},
\textbf{81} (1985), 29--46.

\bibitem[FW]{FW}
B. Fortune, A. Weinstein, 
A symplectic fixed point theorem for complex projective spaces,
\emph{Bull.\ Amer.\ Math.\ Soc.\ (N.S.)}, \textbf{12} (1985),
128--130.

\bibitem[FH]{FrHa}
J. Franks, M. Handel,
Periodic points of Hamiltonian surface diffeomorphisms, \emph{Geom.\
Topol.}, \textbf{7} (2003), 713--756.

\bibitem[Gi]{Gi:conley}
V.L. Ginzburg,
The Conley conjecture, Preprint 2006, math.SG/0610956.

\bibitem[GG1]{GG:CMH}
V.L. Ginzburg, B.Z. G\"urel,
Periodic orbits of twisted geodesic flows and the Weinstein--Moser
theorem, Preprint 2007, arXiv:0705.1818; to appear in  \emph{Comment.\ Math.\
Helv.}

\bibitem[GG2]{GG:gap}
V.L. Ginzburg, B.Z. G\"urel,
Local Floer homology and the action gap, Preprint 2007, arXiv:0709.4077v2.

\bibitem[GG3]{GG:gaps}
V.L. Ginzburg, B.Z. G\"urel,
Action and index spectra and periodic orbits in Hamiltonian dynamics,  
Preprint 2008, arXiv:0810.5170; to appear in \emph{Geom.\ Topol.}

\bibitem[GG4]{GG:generic} 
V.L. Ginzburg, B.Z. G\"urel, 
On the generic
  existence of periodic orbits in Hamiltonian dynamics, in
  preparation.

\bibitem[GGK]{GGK}
V. Guillemin, V. Ginzburg, Y. Karshon, 
\emph{Moment Maps, Cobordisms, and Hamiltonian Group Actions},
Mathematical Surveys and Monographs, \textbf{98}.  American
Mathematical Society, Providence, RI, 2002.

\bibitem[Hi]{Hi}
N. Hingston,
Subharmonic solutions of Hamiltonian equations on tori, Preprint 2004;
to appear in \emph{Ann.\ of Math.}, available at
\url{http://comet.lehman.cuny.edu/sormani/others/hingston.html}.

\bibitem[HS]{HS}
H. Hofer, D. Salamon, 
Floer homology and Novikov rings, in \emph{The Floer memorial volume},
483--524, Progr.\ Math., 133, Birkh\"auser, Basel, 1995.

\bibitem[Ka]{Ka}
Y. Karshon, Private communication, September 2008.

\bibitem[Ki]{Ki}
A.A. Kirillov, \emph{Elements of the Theory of Representations}, 
Springer-Verlag, Berlin, New York, 1976.

\bibitem[MS1]{MS:intro} 
D. McDuff, D. Salamon, 
\emph{Introduction to Symplectic Topology}, The Clarendon Press,
Oxford University Press, New York, 1995.

\bibitem[MS2]{MS}
D. McDuff, D. Salamon,
\emph{J-holomorphic Curves and Symplectic Topology}, Colloquium
publications, vol.\ 52, AMS, Providence, RI, 2004.

\bibitem[Ra1]{Ra1}
H.B. Rademacher, On the average indices of closed geodesics, \emph{J. Differential 
Geom.}, \textbf{29} (1989),  65--83.

\bibitem[Ra2]{Ra2}
H.B. Rademacher,
On a generic property of geodesic flows,  \emph{Math.\ Ann.}, \textbf{298}  (1994),  
101--116. 

\bibitem[SZ]{SZ}
D. Salamon, E. Zehnder,
Morse theory for periodic solutions of Hamiltonian systems and the
Maslov index, \emph{Comm.\ Pure Appl.\ Math.}, \textbf{45} (1992),
1303--1360.

\bibitem[Sc]{Sc2}
M. Schwarz,
A quantum cup-length estimate for symplectic fixed points,
\emph{Invent.\ Math.}, \textbf{133} (1998),  353--397. 

\bibitem[VK]{VK:thesis} O. van Koert, Open books for contact
 five-manifolds and applications of contact homology,
  Inaugural-Dissertation zur Erlangung des Doktorgrades der
  Mathematisch-Naturwissenschaftlichen Fakult\"at der Universit\"at zu
  K\"oln, 2005; available at
\url{http://www.math.sci.hokudai.ac.jp/~okoert/}.

\bibitem[Vi]{Vi}
C. Viterbo, 
Equivariant Morse theory for starshaped Hamiltonian systems,
\emph{Trans.\ Amer.\ Math.\ Soc.}, \textbf{311} (1989), 621--655. 

\bibitem[U]{U}
I. Ustilovsky, Infinitely many contact structures on $S^{4m+1}$,  \emph{
Internat.\ Math.\ Res.\ Notices}  1999,  no. 14, 781--791.

\end{thebibliography}
\end{document}